\documentclass[a4paper,11pt]{article}

\usepackage[a4paper]{geometry}
\usepackage{amsmath,amsthm}
\usepackage{amssymb,amsfonts}
\usepackage{hyperref}
\usepackage{tikz}
\usepackage{enumerate}
\usepackage{graphicx}
\usepackage{calligra}
\usepackage{stmaryrd}
\DeclareSymbolFont{bbold}{U}{bbold}{m}{n}
\DeclareSymbolFontAlphabet{\mathbbm}{bbold}

\title{Generic absoluteness and boolean names for elements of a Polish space}

\theoremstyle{plain}
	\newtheorem{theorem*}{Theorem}
	\newtheorem{theorem}{Theorem}[section]
	\newtheorem{proposition}[theorem]{Proposition}
	\newtheorem{lemma}[theorem]{Lemma}
	
	\newtheorem{corollary}[theorem]{Corollary}
	\newtheorem{claim}{Claim}[theorem]

	\newtheorem{fact}[theorem]{Fact}
\theoremstyle{definition}
	\newtheorem{definition}[theorem]{Definition}
	
	\newtheorem{example}[theorem]{Example}

\theoremstyle{remark}
	\newtheorem{remark}[theorem]{Remark}

\newcommand{\ZFC}{\ensuremath{\textsf{ZFC}}}

\newcommand{\B}{\mathsf{B}}

\newcommand{\C}{\mathbb{C}}

\newcommand{\RO}{{\sf RO}}
\newcommand{\CL}{{\sf CL}}

\newcommand{\cp}[1]{\left( #1 \right)}				

\newcommand{\Qp}[1]{\left\llbracket #1 \right\rrbracket}

\newcommand{\set}[1]{\{#1\}}					
\newcommand{\intclosure}[1]{\mathsf{Reg}(#1)}		

\newcommand{\Int}[1]{\mathsf{Int}\left(#1\right)}
\newcommand{\cl}[1]{\mathsf{Cl}\left(#1\right)}

\renewcommand{\phi}{\varphi} 

\usepackage[mode=multiuser,status=draft,lang=english]{fixme}
\fxsetup{theme=colorsig}
\hypersetup{colorlinks=false, pdfborder={0 0 0}}

\usepackage{titlesec}
\titlelabel{\thetitle. \;}
\titleformat*{\section}{\centering\large\scshape}

\FXRegisterAuthor{M}{aM}{M.Viale}

\author{Andrea Vaccaro, Matteo Viale}
\date{}

\begin{document}
	\maketitle
\begin{abstract}
It is common knowledge in the set theory community that  there
exists a duality relating the commutative $C^*$-algebras
with the family of 
$\B$-names for complex numbers in a boolean valued model for set theory $V^{\B}$.
Several aspects of this correlation have 
been considered in works of the late $1970$'s and early $1980$'s, 
for example 
by Takeuti~\cite{MR555538,MR0505474}, and by
Jech~\cite{MR779056}.
Generalizing Jech's results, we extend this duality so as to be able
to describe the family of boolean names for elements of any given Polish space $Y$ 
(such as the complex numbers) in a boolean valued model for set theory $V^\B$
as a space $C^+(X,Y)$ consisting of functions $f$ whose domain $X$ is the Stone space of $\B$,
and whose range is contained in $Y$ modulo a meager set.
We also outline how this duality can be combined with generic absoluteness results in order
to analyze, by means of forcing arguments, the theory 
of $C^+(X,Y)$.
\end{abstract}

\section{Introduction}
There has been in the early eighties and in the late seventies a wave of attention to the possible 
applications of the forcing machinery in the study of certain type of operator algebras,
key references are Jech's~\cite{MR779056}, and Takeuti's~\cite{MR555538,MR0505474}.
This paper aims to revive this line of research, which in our eyes deserves more attention, at least  from
the set theory community. 
Takeuti and Jech's works outline a correspondence existing between
the theory of commutative unital $C^*\!$-algebras, a specific domain of functional analysis,
and the theory of Boolean valued models, which pertains to logic and set theory.
Takeuti's works~\cite{MR555538,MR0505474} show that one can employ the general machinery of 
forcing to derive certain properties for spaces of operators: this is done first by interpreting these operators as 
suitable $\B$-names for complex numbers in a $\B$-valued model for set theory $V^\B$, then showing
that certain properties can be proved for these $\B$-names using the boolean semantics for $V^\B$, 
finally pulling 
back these properties from the $\B$-names to the corresponding operators.
Jech's~\cite{MR779056} 
develops an algebraic theory of commutative spaces of normal (possibly unbounded) 
operators, in his terminology the
stonean algebras, and proves that
the $\B$-names
for complex numbers in the boolean valued model for set theory $V^{\B}$ can be used to classify up to isomorphism
all possible \emph{complete} stonean algebras. He further develops a functional calculus for stonean algebras and
shows that all the familiar tools given by Gelfand transform for commutative $C^*\!$-algebras naturally 
extend to the framework of stonean algebras. 
In particular~\cite{MR779056} brings to an explicit mathematical form the duality existing
between the theory of commutative $C^*\!$-algebras and the $\B$-names for complex numbers 
in the boolean valued models for set theory of the form $V^\B$. 

This paper generalizes this duality to arbitrary Polish spaces.
We will expand on Jech's and Takeuti's works and devise
a natural translation process to recast the arguments which are used to analyze the 
properties of real numbers in a forcing extension, into arguments that can be expressed 
in the language of functional analysis
enriched with a tiny bit of first order logic.
For example our 
results show how to transform generic absoluteness results, such as Shoenfield's absoluteness and Woodin's
proof of the invariance of the theory of $L(\mathbb{R})$ under set forcing in the presence of class 
many Woodin cardinals~\cite{LARSON,MR959110,MR1713438}, in 
tools to describe the degree of elementarity between the field of complex numbers 
(enriched eventually with Borel predicates) and the ring of germs at points 
of the spectrum
of a commutative $C^* \!$-algebra with extremally disconnected spectrum. 
In particular our results can be seen as a further enhancement of the program launched by 
Takeuti in~\cite{MR555538,MR0505474} aiming to employ forcing methods in the study of operator algebras.
The major outcome of the present paper can be summarized in the following definitions and result:

\begin{quote}
Let $X$ be an extremally disconnected compact Hausdorff space and $Y$ be any Polish space\footnote{
$X$ is extremally disconnected if the closure of an open set is open, or equivalently if its regular open sets are closed.
$Y$ is Polish if it is a separable topological space whose topology can be induced by a complete metric on $Y$.}
with a Polish compactification $\bar{Y}$: for example if $Y=\C$, 
$\bar{Y}=\C\cup\set{\infty}=\mathbb{S}^2$ can be
the one point compactification of $\C$.

Let $C^+(X,Y)$ be the space 
of continuous functions $f:X\to \bar{Y}$ 
such that the preimage of $\bar{Y}\setminus Y$ is 
meager in $X$:
for example if $Y=\C$, $\bar{Y}=\C\cup\set{\infty}=\mathbb{S}^2$, 
we require that $f^{-1}[\set{\infty}]$ is closed nowhere dense in $X$.

For any $p\in X$, let 
$C^+(X,Y)/p$ be the ring of germs in $p$  of functions in $C^+(X,Y)$: i.e. 
$f,g\in C^+(X,Y)$ define the same equivalence class or germ
in $C^+(X,Y)/p$ if $f\restriction U=g\restriction U$ for some open neighborhood $U$ of $p$.

Given $R\subseteq Y^n$ \emph{any} Borel predicate and $p\in X$,
define $R^X/p([f_1],\dots[f_n])$ to hold if there is a neighborhood $U$ of $p$ 
such that
$R(f_1(x),\dots,f_n(x))$ holds on a comeager subset of\footnote{
Recall that $A\subseteq X$ is meager if it is the union of countably many nowhere dense sets, and
$B$ is comeager in $U$ if $U\setminus B$ is meager.
It requires an argument based on the fact that $R$ is Borel
to show that
$R^X/p$ is well defined.} $U$. 
Equivalently $R^X/p$ is the quotient in $p$
of the boolean predicate $R^X:C^+(X,Y)^n \to\mathsf{CL}(X)$ mapping
\[
(f_1,\dots, f_n) \mapsto \intclosure{\set{x\in X: R(f_1(x),\dots,f_n(x))}}
\]
where $\mathsf{CL}(X)$ is the boolean algebra given by 
clopen subsets of $X$, and $\intclosure{A}$ is the interior of the closure of $A$ for any $A\subseteq X$.
\end{quote}
In essence we have defined a sheaf structure on the space of functions $C^+(X,Y)$, which is
making sense not only of the ring of germs at any point of $X$, but also of the interpretation of the Borel predicate
$R$ in such rings of germs. For example, $R$ can be the equality relation on $Y$ or the graph of multiplication
on $\C$, or 
any finite (or countable) combination of such kind of Borel relations on $Y$.

We have the following theorem relating the first order theory of 
the rings of germs $C^+(X,Y)/p$ so defined, to the space $Y$:

\begin{theorem*} \label{thm:names-funct}
Let $X$ be an extremally disconnected compact Hausdorff space and $Y$ be any Polish space.
Fix Borel relations $R_1\subseteq Y^{n_1}$,\dots, $R_k\subseteq Y^{n_k}$.

Then
for all $p\in X$,
the first order structure 
$\langle Y,R_1,\dots,R_n\rangle$ is $\Sigma_2$-elementary in the first order structure 
$\langle C^+(X,Y)/p,R_1^X/p,\dots, R_n^X/p\rangle$.
%

Moreover, if we assume the existence of class many Woodin cardinals, we get that
\[
\langle Y,R_1,\dots,R_n\rangle\prec \langle C^+(X,Y)/p,R_1^X/p,\dots, R_n^X/p\rangle.
\]
\end{theorem*}

Contrary to the case of Jech's and Takeuti's works 
(which require also a certain degree of familiarity with
the basic theory of operator algebras), our results
can be understood by any reader which has  
familiarity with the forcing method and with the basic topological properties of
Polish spaces.
In the case $Y=\C$, the space of functions $C^+(X,Y)$ 
we consider is the unique complete
stonean algebra (according to Jech's terminology of~\cite{MR779056}) whose algebra of projections is given by 
the characteristic functions of clopen sets on $X$. Using Jech's methods, $C^+(X,\C)$ can be 
described as the result of a natural
limit process over the commutative and unital 
$C^*\!$-algebra $C(X)$. It can be seen that one can obtain a different proof of
Theorem~\ref{thm:names-funct} for the case $Y=\C$
using the results of~\cite{MR779056}. We remark nonetheless that the methods
in Jech's paper do not seem to be of use if one aims to give a proof of Theorem~\ref{thm:names-funct} for a
Polish space $Y$ other than $\C$ or $\mathbb{R}$, since his arguments exploit also algebraic features peculiar to
the field structure of $\C$ and to the ordered field $\mathbb{R}$, while our arguments are purely rooted in
the topological properties common to all Polish spaces.

We organize the paper as follows: in section \ref{sec1} we introduce the space of functions $C^+(X,Y)$ with 
$X$ compact, Hausdorff and extremally disconnected, $Y$ Polish,
and we outline its
simplest properties.
In section \ref{section:BVM} we introduce the notion of $\B$-valued model for a first order 
signature, and we show how
to endow $C^+(X,Y)$ 
of the structure of a $\B$-valued model for $\B$ the boolean algebra given 
by regular open (or clopen) sets 
of $X$. In section \ref{section:names} we exhibit a natural isomorphism existing between the $\B$-valued 
models $C^+(X,Y)$ (for $\B$ the clopen sets of $X$) and the family of boolean names for elements of the Polish space 
$Y$ as computed in the boolean valued model for 
set theory $V^\B$.
In section \ref{sec5} we show how to translate generic absoluteness results in 
a proof of Theorem~\ref{thm:names-funct}.
This paper outlines the original parts of the master thesis of the first author \cite{tesi}. A thorough presentation of 
all the results (and the missing details) presented here can be found there.
We try to make the statements of the
theorems comprehensible to most readers with a fair acquaintance with first order logic. 
On the other hand the proofs of our main results will
require a great familiarity with the forcing method.
We encounter a problem in the exposition: those familiar with forcing arguments will find most 
of the proofs redundant or trivial, those unfamiliar with forcing will find the paper far too sketchy. We aim to 
address readers of both kinds, so the current presentation tries to cope with this tension at the best of our possibilities.

\section{The space of functions $C^+(St(\B))$}\label{sec1}
We refer the reader to \cite[Chapter 2]{tesi} or to \cite[Chapter $10$]{halmos}
for a detailed account on the material  
presented in this section.
Let $(X,\tau)$ be a topological space.
For $A\subseteq X$, the interior of $A$ $\textsf{Int}(A)$ is the union of all open sets contained in $A$ and
the closure of $A$ $\textsf{Cl}(A)$ is the intersection of all closed sets containing $A$.
$A$ is regular open if it coincides with the interior of its closure.
$\intclosure{A}=\Int{\cl{(A)}}$ is the regular open set we attach to any $A\subseteq X$.

\begin{itemize}
\item
A topological space $(X,\tau)$ is $0$-dimensional if its clopen sets form a basis for $\tau$.
\item
A compact topological space $(X,\tau)$ is extremally (extremely) disconnected 
if its algebra of clopen sets $\mathsf{CL}(X)$
coincides with its algebra of regular open sets $\mathsf{RO}(X)$.
\end{itemize}

For a boolean algebra $\B$, we let $St(\B)$ be the Stone space of its ultrafilters
with topology generated by the clopen sets
\[
\mathcal{O}_b=\{G\in St(\B):b\in G\}.
\]

The following holds:
\begin{itemize}
\item
$St(\B)$ is a compact $0$-dimensional Hausdorff space, and any
$0$-dimensional compact Hausdorff space $(X,\tau)$ is isomorphic to 
$St(\mathsf{CL}(X))$,
\item
A compact Hausdorff space $(X,\tau)$ is extremally disconnected if and only if
its algebra of clopen sets 
is a complete boolean
algebra.
In particular $St(\B)$ is extremally disconnected if and only if 
$\B=\mathsf{CL}(St(\B))$ is complete.
\end{itemize}
Recall also that the algebra of regular open sets of a topological space $(X,\tau)$ is 
always a complete boolean algebra
with operations
\begin{itemize}
\item
$\bigvee \{A_i:i\in I\}=\intclosure{\bigcup\{A_i:i\in I\}}$,
\item
$\neg A=\Int{X\setminus A}$,
\item
$A\wedge B=A\cap B$.
\end{itemize}

An antichain on a boolean algebra $\B$ is a subset $A$ such that
$a\wedge b=0_\B$ for all $a,b\in A$, $\B^+=\B\setminus\{0_\B\}$ is the family of positive elements of 
$\B$,
and a dense subset of $\B^+$ is a subset $D$ such that for all $b\in \B^+$ there is 
$a\in D$ such that $a\leq_\B b$. In a complete boolean algebra $\B$
any dense subset $D$ of $\B^+$ contains an antichain $A$ such that 
$\bigvee A=\bigvee D=1_\B$, recall also that a predense subset $X$ of $\B$ is a subset such that
$\bigvee X=1_\B$ or equivalently such that its downward closure is dense in $\B^+$.

Another key observation on Stone spaces of complete boolean algebras 
we will often need is the following:
\begin{fact}
Assume $\B$ is a complete atomless boolean algebra, then on its Stone space $St(\B)$:
\begin{itemize}
\item
$\mathcal{O}_{\bigvee_\B A}=\intclosure{\bigcup_{a\in A}\mathcal{O}_a}$ for all $A\subseteq \B$.
\item
$\mathcal{O}_{\bigvee_\B A}=\bigcup_{a\in A}\mathcal{O}_a$ for all finite sets $A\subseteq\B$.
\item
For any infinite antichain $A\subseteq \B^+$, $\bigcup_{a\in A}\mathcal{O}_a$ is properly contained
in $\mathcal{O}_{\bigvee_\B A}$ as a dense open subset 
($\set{\lnot a:a\in A}$ has the finite intersection property and can be extended to an ultrafilter disjoint from $A$).
\end{itemize}
\end{fact}

Given a compact Hausdorff topological space $X$, we let $C^+(X)$ be the space of continuous
functions
\[
f:X\to \mathbb{S}^2=\mathbb{C}\cup\{\infty\}
\] 
(where $\mathbb{S}^2$ is seen as the one point
compactification of $\mathbb{C}$) with the property that $f^{-1}[\{\infty\}]$ is a closed nowhere dense
(i.e. with a dense open complement)
subset of $X$. 
In this manner we can endow $C^+(X)$ of the structure of a commutative
ring of functions with involution,
letting the operations be defined pointwise on all points whose image is 
in $\mathbb{C}$, and be undefined on the preimage of $\infty$.
More precisely $f+g$ is the unique continuous function 
\[
h:X\to \mathbb{S}^2
\] 
such that
$h(x)=f(x)+g(x)$ whenever this makes sense (it makes sense on an open dense subset of 
$X$, since the preimage of the point at infinity under $f,g$ is closed nowhere dense)
and is extended by continuity on the points on which $f(x)+g(x)$ is undefined. 
Thus $f+g\in C^+(X)$
if $f,g\in C^+(X)$. Similarly we define the other operations.
We take the convention that constant functions are always
denoted by their constant value, and that
$0=1/\infty$.
We leave to the reader as an instructive exercise the following:
\begin{lemma}
Let $X$ be compact Hausdorff extremally disconnected.
Then for any $p\in X$ the ring of germs $C^+(X)/p$ is
an algebraically closed field.
\end{lemma}
Its proof will be an immediate corollary of the main theorem we stated in the introduction, since the
theory of algebraically closed fields is axiomatizable by means of $\Pi_2$-formulas using simple Borel predicates
on $\mathbb{C}^n$ for all $n$.
However, as a warm up for the sequel, the reader can try to prove that it is a field.

\begin{remark}
The reader is averted that the spaces of functions $C^+(X)$ we are considering 
may not be exotic: for example let $\nu$ denote the Lebesgue measure on $\mathbb{R}$, and
$\mathsf{MALG}$ denote the complete boolean algebra given by Lebesgue-measurable sets modulo
Lebesgue null sets. $C(St(\mathsf{MALG}))$ is isometric to 
$L^\infty(\mathbb{R})$ via the Gelfand-transform of
the $C^*$-algebra $L^\infty(\mathbb{R})$, and consequently $St(\mathsf{MALG})$ is homeomorphic
to the space of characters of $L^\infty(\mathbb{R})$ endowed with the \emph{weak-$*$ topology}
inherited from the dual of $L^\infty(\mathbb{R})$.
$C^+(St(\mathsf{MALG}))=L^{\infty+}(\mathbb{R})$ is obtained by adding to $L^{\infty}(\mathbb{R})$ the
measurable functions $f:\mathbb{R}\to\mathbb{S}^2$ 
such that $\nu(\bigcap_{n\to\infty}\set{x:|f(x)|>n})=0$.

Moreover by means of the Gelfand transform the spaces 
$C^+(X)$ we are considering are always obtained canonically from a commutative unital 
$C^*$-algebra with extremally disconnected spectrum by a completion procedure as the one described 
above for
$L^{\infty+}(\mathbb{R})$. Jech's~\cite{MR779056} is an useful source for 
those aiming to explore further this analogy.
\end{remark}

\section{Boolean Valued Models} \label{section:BVM}
In a first order model a formula can be interpreted as true or false. 
Given a complete boolean algebra $\B$,
$\B$-boolean valued models generalize Tarski semantics associating to each formula a value in $\B$, 
so that
there are no more only true and false propositions (those associated to 
$1_{\B}$ and $0_{\B}$ respectively),
but also other ``intermediate values'' of truth.
The classic definition of boolean valued models for set theory
and of their semantic
for the language $\mathcal{L}=\set{\in}$ may be found
in \cite[Chapter $14$]{Jech}. As mentioned earlier,
we need to generalize the definition to any first order language and to any theory of the language. 
A complete account of the theory of these boolean
valued models can be found in \cite{Rasiowa}. Since this book is a bit out of date, we recall below the basic facts we will need and we invite the reader to consult \cite[Chapter 3]{tesi} for a detailed account on the material of this section.

\begin{definition} \label{str2}
Given a complete boolean algebra $\B$ and a first order language 
\[ 
\mathcal{L}=\left\{R_i :i\in I \right\} \cup \left\{ f_j:j \in J \right\} 
\]
a \emph{$\B$-boolean valued model} (or $\B$-valued model) 
$\mathcal{M}$ in the language $\mathcal{L}$ is a tuple 
\[ 
\langle M, =^\mathcal{M}, R_i^\mathcal{M} : i \in I, f_j^\mathcal{M} : j \in J\rangle
 \] 
where:
\begin{enumerate}
\item \label{nn1} $M$ is a non-empty set, called \emph{domain} of the 
$\B$-boolean valued model, whose elements are called \emph{$\B$-names};
\item $=^\mathcal{M}$ is the \emph{boolean value} of the equality: 
\begin{align*}
=^\mathcal{M}:M^2 &\to \B \\
(\tau,\sigma) &\mapsto \Qp{\tau=\sigma}^{\mathcal{M}}_{\B}
\end{align*}
\item The forcing relation $R_i^{\mathcal{M}}$ is the \emph{boolean interpretation}
of the $n$-ary relation symbol $R_i$:
\begin{align*}
R_i^{\mathcal{M}}:M^n &\to \B \\
(\tau_1,\dots,\tau_n) &\mapsto \Qp{R_i(\tau_1,\dots, \tau_n)}^{\mathcal{M}}_{\B}
\end{align*}
\item $f_j^{\mathcal{M}}$ is the \emph{boolean interpretation}
of the $n$-ary function symbol $f_j$:
\begin{align*}
f_j^{\mathcal{M}}: M^{n+1} &\to \B \\
(\tau_1,\dots,\tau_n,\sigma) &\mapsto \Qp{f_j(\tau_1,\dots, \tau_n)=\sigma}^{\mathcal{M}}_{\B}
\end{align*}
\end{enumerate}

We require that the following conditions hold:
\begin{itemize}
\item[] for $\tau,\sigma,\chi\in M$,
\begin{enumerate}[(i)]
\item \label{i} $\Qp{\tau=\tau}^{\mathcal{M}}_{\B}=1_{\B}$;
\item $\Qp{\tau=\sigma}^{\mathcal{M}}_{\B}=\Qp{\sigma=\tau}^{\mathcal{M}}_{\B}$;
\item $\Qp{\tau=\sigma}^{\mathcal{M}}_{\B}
\wedge\Qp{\sigma=\chi}^{\mathcal{M}}_{\B} \le \Qp{\tau=\chi}^{\mathcal{M}}_{\B}$;
\end{enumerate}
\item[] for $R\in \mathcal{L}$ with arity $n$, and
$(\tau_1,\dots, \tau_n),(\sigma_1,\dots, \sigma_n) \in M^n$,
\begin{enumerate}[(iv)]
\item  $ ( \bigwedge_{h \in \left\{1, \dots, n  \right\} }
  \Qp{\tau_h=\sigma_h}^{\mathcal{M}}_{\B}  ) \wedge 
  \Qp{R(\tau_1, \dots, \tau_n)}^{\mathcal{M}}_{\B}
    \le \Qp{ R(\sigma_1, \dots, \sigma_n)}^{\mathcal{M}}_{\B} $;
\end{enumerate}
\item[] for $f_j \in \mathcal{L}$ with arity $n$, and
$(\tau_1,\dots, \tau_n),(\sigma_1,\dots, \sigma_n) \in M^n$
and $\mu,\nu \in M$,
\begin{enumerate}[(i)] \setcounter{enumi}{4}
\item $( \bigwedge_{h \in \left\{ 1, \dots, n \right\} } \Qp{ \tau_h = \sigma_h}^{\mathcal{M}}_{\B}  ) 
\wedge  \Qp{f_j(\tau_1, \dots, \tau_n) = \mu }^{\mathcal{M}}_{\B}
 \le  \Qp{f_j(\sigma_1, \dots, \sigma_n) = \mu}^{\mathcal{M}}_{\B} $;
\item \label{vi} $\bigvee_{\mu\in M} \Qp{f_j(\tau_1, \dots, \tau_n) = \mu}^{\mathcal{M}}_{\B} = 1_{\B}$;
\item \label{vii} $ \Qp{f_j(\tau_1, \dots, \tau_n)=\mu}^{\mathcal{M}}_{\B}  
\wedge  \Qp{f_j(\tau_1, \dots, \tau_n)=\nu}^{\mathcal{M}}_{\B} \le \Qp{\mu=\nu}^{\mathcal{M}}_{\B} $.
\end{enumerate}
\end{itemize}
If no confusion can arise, we will omit the subscript $\B$ and the superscript $\mathcal{M}$ 
and we will confuse a function 
or predicate symbol with its interpretation.

Given a $\B$-valued model $\langle M,=^M\rangle$ for the equality, a forcing relation $R$ on $M$ is a
map $R:M^n\to\B$ satisfying condition $(iv)$ above for boolean predicates.  
\end{definition}

We now define the relevant maps between these objects.

\begin{definition}\label{mor}
Let $\mathcal{M}$ be a $\B$-valued model and $\mathcal{N}$ a 
$\mathsf{C}$-valued model in the same language
$\mathcal{L}$. Let
\[  
i: \B \to \mathsf{C}
\]
be a morphism of boolean algebras and $\Phi \subseteq M \times N$ a relation. 
The couple $\langle i, \Phi \rangle$
is a \emph{morphism} of boolean valued models if:
\begin{enumerate}
\item \label{n1} $\text{dom}\Phi=M$;
\item \label{n2} given $(\tau_1,\sigma_1),(\tau_2,\sigma_2) \in \Phi$: 
\[ 
 i(\Qp{\tau_1=\tau_2}^{\mathcal{M}}_{\B}) \le
 \Qp{\sigma_1=\sigma_2}^{\mathcal{N}}_{\mathsf{C}},
 \]
\item \label{n3} given $R$ an $n$-ary relation symbol and
$(\tau_1,\sigma_1),\dots,(\tau_n,\sigma_n) \in \Phi$:
 \[
  i(\Qp{R(\tau_1,\dots,\tau_n)}^{\mathcal{M}}_{\B}) \le
  \Qp{R(\sigma_1,\dots,\sigma_n)}^{\mathcal{N}}_{\mathsf{C}},
  \]
\item \label{n4} given $f$ an $n$-ary function symbol and
$(\tau_1,\sigma_1),\dots,(\tau_n,\sigma_n),(\mu,\nu) \in \Phi$: 
\[
 i(\Qp{f(\tau_1,\dots,\tau_n)=\mu}^{\mathcal{M}}_{\B}) \le
 \Qp{f(\sigma_1,\dots,\sigma_n)=\nu}^{\mathcal{N}}_{\mathsf{C}},
  \]
\end{enumerate}

An \emph{injective morphism} is a morphism such that in \ref{n2} equality holds.

An \emph{embedding} of boolean valued models is an injective morphism such that in
\ref{n3} and \ref{n4} equality holds.

An embedding $\langle i,\Phi \rangle$ from $\mathcal{M}$ to $\mathcal{N}$ 
is called \emph{isomorphism} of boolean valued models
if $i$ is an isomorphism of boolean algebras, and
for every $b \in N$ there is $a \in M$ such that $(a,b)\in \Phi$.

Suppose $\mathcal{M}$ is a $\B$-valued model and 
$\mathcal{N}$ a $\mathsf{C}$-valued model (both in the same
language $\mathcal{L}$) such that $\B$ is a complete subalgebra of $\mathsf{C}$ and
$M\subseteq N$. Let $J$ be the immersion of $\mathcal{M}$ in $\mathcal{N}$.
$\mathcal{N}$ is said to be a 
\emph{boolean extension} of $\mathcal{M}$ if $\langle id_{\B}, J \rangle$ is an
embedding of boolean valued models.
\end{definition}
\begin{remark}\label{standard}
When $\B=\mathsf{C}$ we will consider $i=id_{\B}$ unless otherwise stated.
\end{remark}

Since we are allowing function symbols in $\mathcal{L}$, the definition of the semantic
is a bit more intricate than in the case of a purely relational language.

\begin{definition}\label{BVS}
Given a $\B$-valued model $\mathcal{M}$ in a language $\mathcal{L}$,
let $\phi$ be an $\mathcal{L}$-formula
whose free variables are in $\set{x_1,\dots,x_n}$, and let $\nu$ be a
valuation of the free variables in $\mathcal{M}$ whose
domain contains $\set{x_1,\dots,x_n}$.
We denote with $\Qp{\phi(\nu)}^{\mathcal{M}}_{\B}$ the \emph{boolean value} of $\phi(\nu)$.

First, let $t$ be an $\mathcal{L}$-term and $\tau\in M$; we define recursively
$\Qp{(t=\tau)(\nu)}^{\mathcal{M}}_{\B}\in \B$ as follows:
\begin{itemize}
\item if $t$ is a variable $x$, then 
\[ 
\Qp{(x=\tau)(\nu)}^{\mathcal{M}}_{\B}=\Qp{\nu(x)=\tau}^{\mathcal{M}}_{\B} \]
\item if $t=f(t_1,\dots,t_n)$ where $t_i$ are terms and $f$ is an $n$-ary function symbol, then 
\[
 \Qp{(f(t_1,\dots,t_n)=\tau)(\nu)}^{\mathcal{M}}_{\B}= \bigvee_{\sigma_1,\dots,\sigma_n\in M} 
 \cp{\bigwedge_{1\le i \le n} \Qp{(t_i=\sigma_i)(\nu)}^{\mathcal{M}}_{\B}}   
 \wedge \Qp{f(\sigma_1,\dots,\sigma_n)=\tau}^{\mathcal{M}}_{\B}
  \]
\end{itemize}
Given a formula $\phi$, we define recursively $\Qp{\phi(\nu)}^{\mathcal{M}}_{\B}$ as follows:
\begin{itemize}
\item if $\phi \equiv t_1=t_2$, then 
\[
\Qp{(t_1=t_2)(\nu)}^{\mathcal{M}}_{\B} = \bigvee_{\tau \in M}\Qp{(t_1=\tau)(\nu)}^{\mathcal{M}}_{\B}
 \wedge \Qp{(t_2=\tau)(\nu)}^{\mathcal{M}}_{\B} 
\]
\item if $\phi \equiv R(t_1,\dots,t_n)$, then 
\[ 
\Qp{(R(t_1,\dots,t_n))(\nu)}^{\mathcal{M}}_{\B} = \bigvee_{\tau_1,\dots,\tau_n\in M}
 \cp{\bigwedge_{1\le i \le n} \Qp{(t_i=\tau_i)(\nu)}^{\mathcal{M}}_{\B}} 
 \wedge \Qp{R(\tau_1,\dots,\tau_n)}^{\mathcal{M}}_{\B}
 \]
\item if $\phi\equiv \lnot \psi$, then 
\[ 
\Qp{\phi(\nu)}^{\mathcal{M}}_{\B}=\lnot \Qp{\psi(\nu)}^{\mathcal{M}}_{\B}
\]
\item if $\phi \equiv \psi \wedge \theta$, then
 \[
\Qp{\phi(\nu)}^{\mathcal{M}}_{\B}=\Qp{\psi(\nu)}^{\mathcal{M}}_{\B}
\wedge \Qp{\theta(\nu)}^{\mathcal{M}}_{\B} 
\]
\item if $\phi\equiv \exists y \psi(y)$, then 
\[
 \Qp{\phi(\nu)}^{\mathcal{M}}_{\B}= \bigvee_{\tau\in M} \Qp{\psi(y/\tau,\nu)}^{\mathcal{M}}_{\B}
 \]
\end{itemize}
If no confusion can arise, we omit the index $\mathcal{M}$ and the subscript $\B$, and we simply denote
the boolean value of a formula $\phi$ with parameters in $\mathcal{M}$ by $\Qp{\phi}$.
\end{definition}

By definition, an isomorphism of boolean valued models preserves the boolean value of the 
atomic formulas.
Proceeding by induction on the complexity, one can get the result for any formula.

\begin{proposition}
Let $\mathcal{M}$ be a $\B$-valued model and $\mathcal{N}$ a $\mathsf{C}$-valued model in the same
language $\mathcal{L}$. Assume $\langle i, \Phi \rangle$ is an isomorphism of boolean valued models.
Then for any $\mathcal{L}$-formula $\phi(x_1,\dots,x_n)$, 
and for every $(\tau_1,\sigma_1), \dots, (\tau_n,\sigma_n) \in \Phi$ we have that:
\[
 i(\Qp{\phi(\tau_1,\dots,\tau_n)}^{\mathcal{M}}_{\B})=
 \Qp{\phi(\sigma_1,\dots,\sigma_n)}^{\mathcal{N}}_{\mathsf{C}}
  \]
\end{proposition}

With elementary arguments it is
possible to prove the Soundness Theorem also for boolean
valued models.

\begin{theorem}[Soundness Theorem] \label{soundness}
Assume that $\phi$ is an $\mathcal{L}$-formula which
is syntactically provable by
an $\mathcal{L}$-theory $T$, and that each formula in $T$ has boolean
value at least $b\in \B$ in a $\B$-valued model $\mathcal{M}$.
Then $\Qp{\phi(\nu)}^{\mathcal{M}}_\B\geq b$ for all valuations $\nu$ in $\mathcal{M}$.
\end{theorem}

We get a first order model from a $\B$-valued model passing to a 
quotient by an ultrafilter $G\subseteq \B$. This corresponds for spaces of type 
$C^+(St(\B))$ to a specialization of the space to the
ring of germs in $G$. In the general context it is defined as follows.

\begin{definition} \label{ultraquot}
Let $\B$ a complete boolean algebra, $\mathcal{M}$ a $\B$-valued model in the language 
$\mathcal{L}$, and $G$ an ultrafilter over $\B$.
Consider the following equivalence relation on $M$:
\[ 
\tau \equiv_G \sigma \Leftrightarrow \Qp{\tau=\sigma}\in G 
\]
The first order model
$\mathcal{M}/G =\langle M/G, =^{\mathcal{M}/G}, R_i^{\mathcal{M}/G} : i \in I,
f_j^{\mathcal{M}/G} : j \in J \rangle$ is defined letting:
\begin{itemize}
\item For any  $n$-ary relation symbol $R$ in $\mathcal{L}$
\[
R^{\mathcal{M}/G}=\set{([\tau_1]_G,\dots,[\tau_n]_G)\in (M/G)^n: \Qp{R(\tau_1,\dots,\tau_n)}\in G}.
\]
\item For any  $n$-ary function symbol $f$ in $\mathcal{L}$
\begin{align*}
f^{\mathcal{M}/G}: (M/F)^n &\to M/G \\
([\tau_1]_G,\dots,[\tau_n]_G) &\mapsto [\sigma]_G.
\end{align*}
where $\sigma$ is such that $\Qp{f(\tau_1,\dots,\tau_n)=\sigma} \in G$. 
Def.~\ref{str2}(\ref{vii}) guarantees that this function is well defined.
\end{itemize}
\end{definition}

If we require $\mathcal{M}$ to satisfy a key additional condition, we get an easy
way to control the truth value of formulas in $\mathcal{M}/G$.

\begin{definition}
A $\B$-valued model $\mathcal{M}$ for the language $\mathcal{L}$ is \emph{full} if for every
$\mathcal{L}$-formula $\phi(x,\vec{y})$ and every $\vec{\tau} \in M^{\lvert \vec{y} \rvert}$
there is a $\sigma\in M$ such that
\[
 \Qp{\exists x \phi(x, \vec{\tau})}  =  \Qp{\phi(\sigma, \vec{\tau}) }
 \] 
\end{definition}

\begin{theorem}[Boolean Valued Models \L o\'{s}'s Theorem] \label{los}
Assume $\mathcal{M}$ is a full $\B$-valued model for the language $\mathcal{L}$.
Then for every formula $\phi(x_1, \dots, x_n)$ in $\mathcal{L}$ and $(\tau_1,\dots,\tau_n)\in M^n$: 
\begin{enumerate}[(i)]
\item 
For all ultrafilters $G$ over $\B$,
$\mathcal{M}/G\models \phi([\tau_1]_G,\dots,[\tau_n]_G)$ if and only if
$\Qp{\phi(\tau_1,\dots,\tau_n)}\in G$.
\smallskip

\item
For all $a\in \B$ the following are equivalent:
\begin{enumerate}[(a)]
\item
$\Qp{\phi(f_1,\dots,f_n)}\geq a$,
\item
$\mathcal{M}/G\models \phi([\tau_1]_G,\dots,[\tau_n]_G)$ for all $G\in \mathcal{O}_a$,
\item
$\mathcal{M}/G\models \phi([\tau_1]_G,\dots,[\tau_n]_G)$ for densely many $G\in \mathcal{O}_a$.
\end{enumerate}
\end{enumerate}
\end{theorem}

\subsection{$C^+(St(\B))$ as a boolean valued extension of $\C$}\label{section:ext}

The following example shows how to
obtain a boolean extension of a topological space $Y$ for a
language composed by symbols which are interpreted as Borel subsets
of $Y^n$.

\begin{example}\label{exmp:key}
Fix a complete boolean algebra $\B$ and a topological space $Y$ such that 
\[
\Delta_Y=\set{(x,x)\in Y\times Y: x \in Y}
\]
is Borel
on $Y\times Y$.
Consider $M=C(St(\B),Y)$, the
set of continuous functions from $St(\B)$ to $Y$. 

We define a structure of $\B$-valued
extension
of $Y$ on $M$ for the language with equality as follows.
Given $f,g \in M$, the set
\[
W=\set{G \in St(\B):f(G)=g(G)}
\]
is a Borel subset of $St(\B)$ since both $f$ and $g$ are continuous.
Recall that $A\subseteq Y$ is meager if it is contained in a countable union of closed nowhere dense sets,
and that
$A$ has the Baire property if $U\Delta A$ is meager for some (unique) regular open set $U$.
Since every Borel set $B$ has the Baire property \cite[Lemma 11.15]{Jech}, and $St(\B)$ is compact Hausdorff,
by~\cite[Chapter $29$, Lemma $5$]{halmos}, we get that
\[
\intclosure{\set{G \in St(\B):f(G)=g(G)}}
\]
is the unique regular open set with a meager symmetric difference with $W$. 
Identifying $\B$ with
$\RO(St(\B))$ ($\B$ is complete), we have that 
\[
=^{St(\B)}(f,g)=\Qp{f=g}^{St(\B)}=\intclosure{\set{G \in St(\B):f(G)=g(G)}}
\]
is a well defined element of $\B$ and satisfies the clauses of Def. \ref{str2} for the equality relation.
For any Borel $R\subseteq Y^n$, the predicate $R^{St(\B)}:C(St(\B),Y)^n\to\B$ defined by
\[
R^{St(\B)}(f_1,\dots,f_n)=\Qp{R(f_1,\dots,f_n)}^{St(\B)}=\intclosure{\set{G\in St(\B):R(f_1(G),\dots,f_n(G))}}
\]
is a forcing relation $R$ satisfying the clauses of 
Def.~\ref{str2}.
Similarly we can lift Borel functions $F:Y^n\to Y$.

With these definitions it can be checked that 
\[
\mathcal{M}=\langle C(St(\B),Y),=^{St(\B)},R_i^{St(\B)}:i\in I, F_j^{St(\B)}:j\in J\rangle
\] 
is a $\B$-valued model for the signature
given by the Borel relations $\langle R_i:i\in I\rangle$ and Borel functions $\langle F_j:j\in J\rangle$ chosen on $Y$.
Moreover the set $\set{c_x \in M:x \in Y}$, where $c_x$ is the constant
function with value $x$, is a copy of $Y$ in $M$, i.e: the complete homomorphism given by the inclusion of
$2$ in $\B$ induces an embedding of the $2$-valued model $\langle Y,=,R_i:i\in I, F_j:j\in J\rangle$ 
into the $\B$-valued model $\mathcal{M}$
mapping $x\mapsto c_x$ (however we do not \emph{as yet} assert that this embedding preserves the truth of
formulas with quantifiers).
Thus we can infer that $\mathcal{M}$ is a $\B$-valued extension of an isomorphic
copy of $Y$ seen as a
$2$-valued structure in a relational language with relation symbols interpreted
as Borel subsets of $Y^n$.

Finally if $G$ is an ultrafilter on $St(\B)$, i.e. a point of $St(\B)=X$, we can define the ring $C(X,Y)/G$ of germs in
$C(X,Y)$ letting 
\[
[f]_G=\set{g:\Qp{f=g}^{St(\B)}\in G}
\] 
and $R^{St(\B)}([f_1]_G,\dots,[f_n]_G])$ iff
$R^{St(\B)}(f_1,\dots,f_n)\in G$. We can easily check that the map
$x\mapsto [c_x]_G$ defines an embedding of $2$-valued models of $\langle Y,=,R_i:i\in I, F_j:j\in J\rangle$ into
$\mathcal{M}/G$.

If $Y$ is Polish (i.e. second countable and completely metrizable) $\Delta_{Y}$ is closed
($Y$ is Hausdorff), therefore, for any fixed language $\mathcal{L}$
whose elements are Borel relations and functions on $Y$, we can define
a structure of $\B$-valued extension of $Y$ for the language $\mathcal{L}$.
If $Y=\C$, the domain of such extension is the $C^* \!$-algebra $C(St(\B))$ with extremally disconnected spectrum.

It can be checked that if $Y$ is compact,
$C(St(\B),Y)$ endowed with suitable lifting of Borel predicates 
is a full $\B$-valued model, while if
$Y$ is not compact and contains an infinite set with discrete relative topology (e.g. $\mathbb{N}$ as a subset of $\C$),
$C(St(\B),Y)$ is not a full $\B$-valued model (see Remark \ref{rmk:C+xvsCX} below). 
\end{example}

The latter observation is one of the compelling reasons which lead us to
associate to
$\C$ (which is Polish non-compact, locally compact)
the space of functions $C^+(St(\B))$ (which we show to be a \emph{full} $\B$-valued model). 
Similar tricks will be needed to properly describe the full boolean extensions
of arbitrary (non-compact) Polish spaces by means of spaces of functions.

We resume the above observations in the following definition:
\begin{definition}\label{def:C+X}
Let $X$ be a compact Hausdorff extremally disconnected topological space.

\begin{enumerate}[(i)]
\item Let $Y$ be a topological space such that $\Delta_Y$ is Borel in $Y^2$. 
For any Borel relation $R$ on $Y^n$, $R^X:C(X,Y)^n\to \CL(X)$ maps
$(f_1,\dots,f_n)$ to the clopen set
\[
\intclosure{\set{G\in X: R(f_1(G),\dots,f_n(G))}}.
\]
The lifting of Borel functions on $Y$ to $C(X,Y)$ is obtained by lifting their graph to a forcing relation
on $C(X,Y)$.
\item
We let $C^+(X)$ be the space of continuous
functions
\[
f:X\to \mathbb{S}^2=\mathbb{C}\cup\{\infty\}
\] 
(where $\mathbb{S}^2$ is seen as the one point
compactification of $\mathbb{C}$) with the property that $f^{-1}[\{\infty\}]$ is a closed nowhere dense
subset of $X$.
We lift Borel relations $R\subseteq \mathbb{C}^n$ to $R^{X}$ again letting 
\[
R^X(f_1,\dots,f_n)=\intclosure{\set{G\in X: R(f_1(G),\dots,f_n(G))}}.
\]
\end{enumerate}
We let $\langle C(X)/G,R^X/G\rangle$ and $\langle C^+(X)/G,R^X/G\rangle$ be the associated ring of germs with
$R^X/G$ defined for both rings by the requirement: $R^{X}([f_1]_G,\dots,[f_n]_G])$ iff
$G \in R^{X}(f_1,\dots,f_n)$.
 
\end{definition}

We have the following Lemmas:
\begin{lemma}[Mixing Lemma]  \label{mixing:lemma}
Assume $\B$ is a complete boolean algebra and
$A\subseteq\B$ is an antichain.
Then for all families $\{f_a:a\in A\}\subseteq C^+(St(\B))$, there exists 
$f\in C^+(St(\B))$ such that 
\[
a\leq\llbracket f=f_a\rrbracket
\]
for all $a\in A$.
\end{lemma}
\begin{proof}
Sketch:
Let $f\in C^+(St(\B))$ be the unique function such that
$f\restriction \mathcal{O}_{(\neg\bigvee A)}=0$ and 
$f\restriction \mathcal{O}_a=f_a\restriction \mathcal{O}_a$ for all $a\in A$. Check that $f$ is well defined and works.
\end{proof}
\begin{lemma}[Fullness Lemma] \label{fullness:lemma}
Assume $\B$ is a complete boolean algebra.
Let $R_1,\dots,R_n$ be forcing relations on $C^+(St(\B))^{<\mathbb{N}}$.
Then for all formulas $\phi(x,\vec{y})$ in the language $\{R_1,\dots,R_n\}$ 
and all $\vec{f}\in C^+(St(\B))^n$, there exists $g\in C^+(St(\B))$ such that
\[
\Qp{\exists x\phi(x,\vec{f})}=\Qp{\phi(g,\vec{f})}.
\]
\end{lemma}
\begin{proof}
Sketch:
Find $A$ maximal antichain among the $b$ such that 
$\Qp{\phi(g_b,\vec{f})}\geq b$ for some $g_b$. Now apply the Mixing Lemma
to patch together all the $g_a$ for $a\in A$ in a $g$.
Check that
\[
\Qp{\exists x\phi(x,\vec{f})}=\Qp{\phi(g,\vec{f})}.
\]
\end{proof}
%
%

\section{$\B$-names for elements of a Polish space} \label{section:names}
We refer the reader to \cite{Jech} for a comprehensive treatment of the forcing method, and to 
\cite[Chapter 3]{tesi} for a sketchy presentation covering in more detail the results of this section.
All over this section we assume the reader has some familiarity with the standard presentations of forcing 
and we follow notation standard in the set theoretic community 
(for example $\mathbb{N}$ is often denoted as the ordinal $\omega$).
Throughout this section we will assume $V$ (the universe of sets) to be a transitive model of $\ZFC$,
and $\B\in V$ a boolean algebra which $V$ models to be complete.
$V^{\B}$ will denote the boolean valued model of set theory
as defined in \cite[Chapter $14$]{Jech} and $\check{a}\in V^{\B}$ will denote the canonical $\B$-names for
sets $a\in V$.
If $G$ is a $V$-generic ultrafilter in $\B$, $V[G]$ will denote the generic
extension of $V$ and $\sigma_G$ the interpretations of $\B$-names in $V^\B$ by $G$.
In this situation there is a natural isomorphism between $(V^\B/G,\in^\B/G)$ and
$(V[G],\in)$ defined by $[\sigma]_G\mapsto\sigma_G$.
Cohen's forcing theorem in this setting states the following for any formula
$\phi(x_1,\dots,x_n)$ in the language of set theory:
\begin{itemize}
\item $V[G]\models\phi((\sigma_1)_G,\dots,(\sigma_n)_G)$ if and only if
$\Qp{\phi(\sigma_1,\dots,\sigma_n)}\in G$,
\item
$\Qp{\phi(\sigma_1,\dots,\sigma_n)}\geq b$ if and only if
$V[G]\models\phi((\sigma_1)_G,\dots,(\sigma_n)_G)$ for all $V$-generic filters $G$ to which $b$ belongs.
\end{itemize}
It is well known that $V$-generic filters cannot exist for atomless complete boolean algebra, nonetheless
there is a wide spectrum of solutions to overcome this issue and work as if for
any such algebra $V$-generic filters can be found, and we will do so.
We will also use in several points the following form of absoluteness for $\Delta_1$-properties.
\begin{quote}
Let $c$ denote a new constant symbol. 
Then for all provably $\Delta_1$-definable properties $\phi(x,c)$ over the
theory $\ZFC+(c\subseteq\omega)$ and all\footnote{Recall that
$H_{\omega_1}$ is the family of hereditarily countable sets. For what concerns us, the relevant observation is that
any Polish space is a definable class (with parameters) in $H_{\omega_1}$.} $a\in H_{\omega_1}$, the following holds:
\begin{itemize} 
\item
$\phi(x/a,c/r)$ holds in a transitive $N$ which is a model of (a large enough fragment of) $\ZFC$ with $a,r\in N$
if and only if 
$\Qp{\phi(x/\check{a},c/\check{r})}=1_\B$ holds in $N$ for all boolean algebras $\B\in N$ 
which $N$ models to be complete.
\item
$\phi(x/a,c/r)$ holds in $V$ if and only if it holds in any (some) transitive set $N$ 
which is a model of (a large enough fragment of) $\ZFC$ with $a,r\in N$.
\end{itemize}
\end{quote}

Let $Y$ be a Polish space. Then $Y$ can be identified with a
$G_\delta$-subset of the Hilbert cube $\mathcal{H}=[0,1]^{\mathbb{N}}$ \cite[Theorem $4.14$]{Kechris}.

Consider 
\[
\hat{\mathcal{B}}=\set{B_r(q):r\in\mathbb{Q}, q\in D}
\]
where $B_r(q)$ is the open ball of radius $r$ and center $q$, and $D$ is the set of points in
$\mathcal{H}$ with rational coordinates which are non-zero just on a finite set.
Then $\hat{\mathcal{B}}$ is a countable basis for the topology on $\mathcal{H}=[0,1]^{\mathbb{N}}$,
and it is
described by a provably $\Delta_1$-definable property defined by a lightface Borel predicate.

\begin{definition} Let $Y$ be a Polish space in $V$, w.l.o.g. 
\[
Y=\bigcap_{n\in\mathbb{N}}\bigcup\set{B_{r_{mn}}(q_{mn}):m\in\mathbb{N}}
\]
is a $G_\delta$-subset of $\mathcal{H}$
given by a suitably chosen family of elements $B_{r_{mn}}(q_{mn})$ of $\hat{\mathcal{B}}$.
$\sigma \in V^{\B}$
is a \emph{$\B$-name for an element of $Y$} if
\[
\Qp{\sigma\in \bigcap_{n\in\mathbb{N}}\bigcup\set{\dot{B}_{r_{mn}}(q_{mn}):m\in\mathbb{N}}}=1_{\B},
\]
where $(\dot{B}_{r}(q))_G$ is in $V[G]$ the ball of radius $r$ and center $q$ of the space 
$\mathcal{H}$ as defined in $V[G]$
for all $V$-generic filters $G$.

We denote by $Y^{\B}$ the set of all $\B$-names (of minimal rank) for elements of $Y$ modulo
the
equivalence relation:
\[
\sigma \equiv \tau \Leftrightarrow \Qp{\sigma = \tau}=1_{\B}
\]
We will call $\B$-name for a complex number any element of the family $\C^\B$.
\end{definition}

We can similarly lift Borel relations on $Y^n$ to boolean relations on $(Y^\B)^n$:
\begin{remark} \label{code}
Let $Y$ be a Polish space. 
As already noted, $Y$ is a $G_\delta$-subset
of $\mathcal{H}=[0,1]^{\mathbb{N}}$.
$\hat{\mathcal{B}}$ induces a countable open basis
on $Y$:
\[
\hat{\mathcal{B}}_Y=\set{B_r(q)\cap Y:r\in\mathbb{Q}, q\in D}.
\]
Every Borel subset of $Y$ is obtained,
in fewer than $\aleph_1$ steps, from
the elements of $\hat{\mathcal{B}}_Y$ by taking
countable unions and complements.
It is possible to code these operations with $r$ a subset of $\omega$ (see \cite[Chapter $25$]{Jech}).
For our purposes it is enough
to say that if $R$ is a Borel subset of $Y^n$,
there is some $r \subseteq\omega$ and a ($\ZFC$ provably) $\Delta_1$-property $P_R(\vec{x},y)$ such that
\[
\vec{x} \in R \Leftrightarrow P_R(\vec{x},r).
\]
Suppose $r \in V$. We denote by $R^V$ the set $\set{\vec{x}\in V: P_R(\vec{x},r)}$.
\end{remark}
Guided by these considerations,
we define in $V$ the following.
\begin{definition} \label{complexinter}
Given $R$, a
Borel $n$-ary relation on a Polish space $Y$, we let 
$P_R(\vec{x},r)$ be the provably $\Delta_1$-definable property such that
\[
\vec{x} \in R \Leftrightarrow P_R(\vec{x},r).
\]
For any
$\sigma_1,\dots,\sigma_n \in Y^{\B}$, let
$\vec{\sigma}\in V^\B$ denote the canonical name for the
tuple $(\sigma_1,\dots,\sigma_n)$.

Define
\[
R^{\B}(\sigma_1,\dots,\sigma_n)=
\Qp{P_R(\vec{\sigma},\check{r})}^{V^\B}.
\]
\end{definition}
Similarly define the lifting to $Y^\B$ of Borel functions $F:Y^n\to Y$.

With these definitions
\[
\langle Y^\B, R_1^{\B},\dots,
R_k^{\B},
F_1^{\B},\dots,
F_l^{\B}\rangle
\] 
is a $\B$-valued extension of $Y$,
where each $R_i$ ($F_j$) is an arbitrary Borel relation (function)
on $Y^{n_i}$ (from $Y^{m_j}$ to $Y$).

\begin{remark}\label{rmk:C+xvsCX} 
So far we have defined a structure of $\B$-valued model for Borel relations and functions on both
$Y^{\B}$ and $C(St(\B),Y)$ for a Polish space $Y$. However, whenever $Y$ is not compact, 
we cannot exhibit a natural isomorphism between these two
models, unless we enlarge $C(St(\B),Y)$.
The problem (that can be appreciated by the reader familiar with forcing) is the following:
assume we split a complete atomless boolean algebra $\B$ in a countable maximal antichain 
$A=\set{a_n:n\in\omega}$. Then $\bigvee_{n\in\omega} a_n=1_\B$ but
$\bigcup_{n\in\omega} \mathcal{O}_{a_n}$ is just an open dense subset of $St(\B)$, as the family
$\{\lnot a_n:n\in\omega\}$ has the finite intersection property and can be extended to an ultrafilter $H$
missing the antichain $A$. Now consider for $Y=\C$ the function $f:G\mapsto n$ iff $a_n\in G$.
This should naturally correspond to the $\B$-name for a natural number 
\[
\sigma_f=\set{\langle \check{m}, a_n\rangle: m<n\in\omega}.
\]
Notice also that the function is continuous on its domain 
since the target is a discrete subspace of $\C$ and the preimage of
each point is clopen. Moreover this function naturally extends to a continuous function in 
$C^+(St(\B))\setminus C(St(\B))$ mapping the $G$ out of its domain to $\infty$. This 
shows that $C(St(\B))$ is a space of functions too small to capture
all possible $\B$-names for complex numbers.
The reader who has grasped the content of this remark will find the proofs of the following Lemmas almost self-evident,
 however we decided to include them in full details, since at some points there are 
 delicate issues regarding the way to formulate certain simple properties of Polish 
 spaces in an absolute (i.e $\Delta_1$-definable) manner,
which can be tricky for  those who are not fully familiar with forcing.
\end{remark}

\begin{definition} \label{Cplus}
Let $Y$ be a Polish space presented as a
$G_\delta$-subset of the Hilbert cube $\mathcal{H}=[0,1]^{\mathbb{N}}$.
Let $\B$ be a complete boolean algebra.

$C^+(St(\B),Y)$
is the family of continuous functions $f:St(\B)\to \mathcal{H}$ 
such that $f^{-1}[\mathcal{H}\setminus Y]$ is meager in $St(\B)$.
\end{definition}

We can define a structure of $\B$-valued extension of $Y$ over $C^+(St(\B),Y)$ repeating
verbatim what we have done in Section \ref{section:ext} for $C(St(\B),Y)$. Everything will work smoothly since
for all Borel $R\subseteq Y^n$ and $f_1,\dots,f_n\in C^+(St(\B),Y)$, the
set of $H\in St(\B)$ such that $R(f_1(H),\dots,f_n(H))$ is not defined is always a meager subset of $St(\B)$.
Moreover Lemmas \ref{mixing:lemma} and \ref{fullness:lemma} can be recasted verbatim also for $C^+(St(\B),Y)$,
which is therefore a full $\B$-valued model. 
We are ready to
prove the following theorem.

\begin{theorem} \label{thm:iso}
Let $Y$ be a Polish space and $\B$ a complete boolean algebra.

Then
$\langle C^+(St(\B),Y),=^{St(\B)}\rangle$ 
and 
$\langle Y^{\B},=^{\B}\rangle$ 
are isomorphic $\B$-valued models.
\end{theorem}

Since the case $Y=\C$ outlines already the main ingredients of the proof and may be slightly easier
to follow, due to the evident analogies of the spaces $C^+(X,\C)$
with commutative $C^*\!$-algebras, we will give the full proof
of the theorem above for this special case. However, with minimal modifications, the reader
will be able to generalize by himself the proof to any Polish space: for spaces admitting a one point compactification
it suffices to replace all occurrences of $\C$ with $Y$ in the proof to follow. For other Polish 
spaces $Z$ not admitting such a simple compactification, this is slightly more delicate
since the preimage of an $f\in C^+(St(\B),Z)$ of 
the points in the range of $f$ out of $Z$ is not anymore a closed nowhere dense
set, but a countable union of closed nowhere dense sets of $St(\B)$. However no essential new 
complications arise also for this case, so we feel free to sketch just the main ingredients of 
the proof for the more general case of such Polish spaces 
$Z$.

\begin{remark}
In the following, given a complete boolean algebra $\B$, we will
often confuse it with $\RO(St(\B))$. If $U$ is a regular open set
of $St(\B)$ and $G\in St(\B)$, we may write equivalently
\[
G \in U,
U\in G 
\]
depending on whether we are considering
$U$ as an element of $\RO(St(\B))$
or as the correspondent element in $\B$.
\end{remark}
\begin{remark} \label{spec}
The definitions given in Remark \ref{code} and Definition \ref{Cplus}
can be simplified when working in $\C$.
Instead of $\hat{\mathcal{B}}_{\C}$ from Remark \ref{code},
we will work directly with $\mathcal{B}=\set{U_n:n \in \omega}$,
the countable basis of $\C$ whose elements are the open balls with
rational radius and whose centre has rational coordinates.
Moreover, instead of Definition \ref{Cplus}, we work with
$C^+(St(\B))$ as defined in Def.~\ref{def:C+X}(ii).
\end{remark}

\subsection*{Proof of Theorem  \ref{thm:iso} for $\mathbb{C}$}
The proof splits in several Lemmas.

The first Lemma gives a characterization of the $\B$-name to associate to an $f\in C^+(St(\B))$,
which we will need in order to define the boolean isomorphism we are looking for.
\begin{lemma} \label{extfunction}
Assume $f \in V$ is an element of $C^+(St(\B))$.
For $H\in St(\B)$ we define
\[
\Sigma^H_f=\set{\cl{U_n}:
\intclosure{f^{-1}[U_n]} \in H}
\]
Then,
for $H\in St(\B)$, we have:
\[
f(H)=\sigma_f^H
\]
where $\sigma_f^H$
it is the unique element in $\bigcap\Sigma_f^H$ if $\Sigma_f^H$ is
non-empty, and $\sigma_f^H=\infty$ otherwise.
\end{lemma}

\begin{remark}\label{rmk:keyrmk}
The Lemma shows that in $\ZFC$, given $f\in C^+(St(\B))$,
it holds that
\[
f(H)=x \Leftrightarrow x=\sigma_f^H.
\]
The latter is a ($\ZFC$ provably) $\Delta_1$-property with $\omega,\B$, and
$\{a_n=\intclosure{f^{-1}[U_n]}:n\in\mathbb{N}\}$ as parameters.
Thus, given
$V$ a transitive model of $\ZFC$, $\B$ a complete boolean algebra in $V$,
$G$ a $V$-generic filter in $\B$, any $f\in V$ element of
$C^+(St(\B))^V$ can be extended in an absolute manner to
$V[G]$ by the rule:
\begin{align*}
f^{V[G]}:St(\B)^{V[G]} &\to \C^{V[G]}\\
H &\mapsto \sigma_f^H
\end{align*}
where $\sigma_f^H$ is defined as in the previous lemma through the set
$
\Sigma_f^H=\set{\cl{U_n}: a_n \in H}
$.
\end{remark}
This observation is used in
the following proposition defining the boolean isomorphism
between $\C^{\B}$ and $C^+(St(\B))$.
\begin{proposition} \label{prop:iso}
Fix $V$ a transitive model of $\ZFC$ and $\B\in V$ a boolean algebra which 
$V$ models to be complete.
Let $f \in C^+(St(\B))$ and consider
\[
\mathcal{B}=\set{U_n: n \in \omega}
\]
the countable basis of $\C$ defined in Remark \ref{spec}. For
each $n\in\omega$ let
\[
a_n=\intclosure{f^{-1}[U_n]}.
\]
There exists a unique $\tau_f \in \C^{\B}$ such that\footnote{$\dot{U_n}$ denotes the $\B$-name for the complex
numbers in the open ball of the generic extension
determined by the rational coordinates and rational radius of the ball $U_n$.}
\[
\Qp{\tau_f \in \dot{U}_n}^{V^\B}=a_n
\]
for all $n\in\omega$.
Moreover any $\tau\in V^\B$ such that 
$\Qp{\tau\in \dot{U}_n}^{V^\B}=a_n$
for all $n\in\omega$ is also such that $\Qp{\tau=\tau_f}^{V^\B}=1_\B$.
\end{proposition}

By Proposition \ref{prop:iso} we conclude that the map $f\mapsto\tau_f$ defines a function between 
$C^+(St(\B))$ and $\C^\B$. 
We still need to show that the function is a surjective boolean map i.e. it maps boolean equality on
$C^+(St(\B))$ to boolean equality on $\C^\B$ and is surjective (in the sense of boolean embeddings).
The latter is achieved by the following Lemma:
\begin{lemma}\label{surj}
Assume $\tau \in \C^{\B}$. Consider
\begin{align*}
f_{\tau}:St(\B) &\to \C\cup\set{\infty} \\
H &\mapsto \sigma_\tau^H
\end{align*}
where, given
\[
\Sigma_\tau^H=\set{\cl{U_n}:\Qp{\tau\in \dot{U}_n}^{V^\B}\in H},
\]
$\sigma_\tau^H $ is the unique element
in $\bigcap \Sigma_\tau^H$ if $\Sigma_f^H$ is non-empty,
$\sigma_\tau^H=\infty$ otherwise.
The function $f_\tau$ belongs to $C^+(St(\B))$
and $\tau_{f_{\tau}}=\tau$.
\end{lemma}

Finally we need to show that $f\mapsto\sigma_f$ respects boolean equality, i.e. that:
\begin{equation}\label{eqn:ZZZ}
\Qp{f=g}^{C^+(St(\B))}=\Qp{\tau_f=\tau_g}^{V^\B}.
\end{equation}

Since it makes no difference to prove the equality for this relation or for an arbitrary Borel 
relation (or functions), we will prove the following stronger result:
\begin{lemma}\label{lem:embed}
Assume $R\subseteq \C^n$ is a Borel relation. Then 
$R^{St(\B)}(f_1,\dots,f_n)=R^\B(\sigma_{f_1},\dots,\sigma_{f_n})$,
where $R^\B$ is defined according to Def.~\ref{complexinter}.
\end{lemma}

It is clear that these Lemmas entail the conclusion of the theorem. We prove all of them in the next subsection.

\begin{corollary} \label{extension}
Under the hypotheses of Proposition~\ref{prop:iso}, if $G$ is a $V$-generic filter
in $\B$ then:
\[
f^{V[G]}(G)=(\tau_f)_G.
\]
\end{corollary}
\begin{proof}
By Lemma~\ref{lem:embed} and Remark~\ref{rmk:keyrmk}, we get that $\tau^H_f=f^{V[G]}(H)$ for all 
$H\in St(\B)^{V[G]}$. Moreover whenever $H$ is $V$-generic for $\B$ we also have that $\tau^H_f=(\tau_f)_H$.
Since $G\in V[G]$ is $V$-generic for $\B$, the conclusion follows.
\end{proof}

\subsection*{Proof of the key Lemmas}
\begin{proof}[Proof of Lemma~\ref{extfunction}]
Assume $\Sigma_f^H$ is empty. If $f(H)\in U_n$ for some $n\in \omega$
it follows that:
\[
H\in f^{-1}[U_n] \subseteq \intclosure{f^{-1}[U_n]}
\]
hence $\cl{U_n} \in \Sigma_f^H$, which is absurd.
Suppose now that $\Sigma_f^H$ is non-empty.
\begin{claim} \label{claim}
Assume $\Sigma_f^H$ is non-empty. Then $\bigcap \Sigma_f^H$ is a singleton.
\end{claim}
\begin{proof}
Let $m\in \omega$ be such that $\cl{U_m} \in \Sigma_f^H$.
\begin{description}
\item[\rm \underline{Existence:}] The family
\[
\hat{\Sigma}_f^H=\set{\cl{U_m} \cap \cl{U_n}:
\intclosure{f^{-1}[U_n]}\in H}
\]
is a family of closed subsets of $\cl{U_m}$. $\Sigma_f^H$
inherits the finite intersection property from $H$, hence so does
$\hat{\Sigma}_f^H$. Since $\cl{U_m}$ is compact, we can conclude that
\[
\emptyset \not= \bigcap \hat{\Sigma}_f^H \subseteq \bigcap
\Sigma_f^H
\]
\item[\rm \underline{Uniqueness:}] Suppose there
are two different points $x,y \in \bigcap \Sigma_f^H$. There exists
$p \in \omega$
such that $x\in U_p, y\notin \cl{U_p}$. The last relation guarantees that
$\cl{U_p}\notin \Sigma_f^H$. Now we show that for
$w\in\bigcap \Sigma_f^H$, $w \in U_n$ implies
$\intclosure{f^{-1}[U_n]}\in H$. Therefore $x\in U_p$ implies
$\cl{U_p}\in \Sigma_f^H$, which is absurd.
Suppose $\intclosure{f^{-1}[U_p]}\notin H$,
we have that: 
\[
H \in
 \intclosure{f^{-1}[U_p]}^c \cap
\intclosure{f^{-1}[U_m]}  \subseteq
f^{-1}[\cl{U_m} \setminus U_p]
\]
For each $z \in \cl{U_m} \setminus U_p$ there exists $U_{n_z}$ such
that
\[
z \in U_{n_z}
\wedge
x\notin\cl{U_{n_z}}
\]
This family of open
balls covers the
compact space $\cl{U_m} \setminus U_p$, so that there are
$z_1, \cdots, z_k \in \cl{U_m} \setminus U_p$ which verify the
following chain of inclusions:
\[
f^{-1}[\cl{U_m} \setminus U_p] \subseteq \bigcup_{1 \le i \le k} 
f^{-1}[U_{n_{z_i}}]
\subseteq \bigcup_{1 \le i \le k} \intclosure{f^{-1}[U_{n_{z_i}}]}
\]
There is therefore a $z_j$ such that
$\intclosure{f^{-1}[U_{n_{z_j}}]}\in H$, hence
$\cl{U_{z_j}}\in \Sigma_f^H$. This is absurd
since $x \notin \cl{U_{z_j}}$.
\end{description}
\end{proof}
Suppose $f(H)\not= \sigma_f^H$
and consider two open balls $U_1,U_2$ in $\mathcal{B}$
such that
\[
\cl{U_1} \cap \cl{U_2} =\emptyset
\]
\[
f(H)\in U_1
\]
\[
\sigma_f^H\in U_2
\]
It easily follows that both $\intclosure{f^{-1}[U_1]}$ and
$\intclosure{f^{-1}[U_2]}$ are in $H$ (the second assertion
can be shown along the same lines of the uniqueness proof in Claim \ref{claim}).
These two sets are disjoint, a contradiction follows.

The Lemma is proved.
\end{proof}

In order to prove Proposition \ref{prop:iso},  we need to generalize
what we have exposed in Remark \ref{code} about Borel codes. 
In particular we need to be able to describe what is the lift of an open (closed) set of $St(\B)$ to the
corresponding open (closed) set in $St(\B)^{V[G]}$ where $G$ is $V$-generic for $\B$.
%
The following can be shown starting from the clopen sets and then extending the proof to cover
the case of arbitrary open or closed sets.
\begin{fact} \label{borel}
Let $G$ be a $V$-generic
filter over $\B$.
Assume $R^V,S^V$ are two open or closed sets in $St(\B)^V$. Then
\[
R^V \subseteq S^V \Leftrightarrow R^{V[G]}\subseteq S^{V[G]}
\]
\end{fact}
\begin{proof}
We deal with the case for open sets, the case for closed sets is proved along the same lines.
Let in $V$, $R=\bigcup_{i\in I}\mathcal{O}_{a_i}$ and $S=\bigcup_{j\in J}\mathcal{O}_{a_j}$.
Now set in $V[G]$,
$R^{V[G]}= \bigcup_{i\in I}\mathcal{O}_{a_i}^{V[G]}$ and $S^{V[G]}= \bigcup_{j\in J}\mathcal{O}_{a_j}^{V[G]}$.
Then $R\subseteq S$ holds in $V$ (or $R^{V[G]}\subseteq S^{V[G]}$ holds in $V[G]$) iff for all $i\in I$
$\mathcal{O}_{a_i}^V\subseteq \bigcup_{j\in J}\mathcal{O}_{a_j}^V$
($\mathcal{O}_{a_i}^{V[G]}\subseteq \bigcup_{j\in J}\mathcal{O}_{a_j}^{V[G]}$). 
By compactness, since $\mathcal{O}_{a_i}^V$ ($\mathcal{O}_{a_i}^{V[G]}$) is a clopen subset of 
$St(\B)^V$ in $V$
(a clopen subset of 
$St(\B)^{V[G]}$ in $V[G]$),
 there is a finite set $J_i\subseteq J$
such that $\mathcal{O}_{a_i}\subseteq \bigcup_{j\in J_i}\mathcal{O}_{a_j}$.
This occurs (both in $V$ or $V[G]$) if and only if $a_i\leq \bigvee_{j\in J_i} a_j$.

Now notice that for any finite set $J_i$,
$V[G]\models a_i\leq \bigvee_{j\in J_i} a_j$ iff $V\models a_i\leq \bigvee_{j\in J_i} a_j$.

We get the thesis.
\end{proof}

\begin{proof}[Proof of Proposition \ref{prop:iso}]
Consider the $\B$-name
\[
\Sigma_f=\set{(\dot{U}_n,a_n): n \in \omega},
\]
where $a_n=\intclosure{f^{-1}[U_n]}$.
Standard forcing arguments give that
\begin{equation}\label{eqn:XXX}
\Qp{\exists! x (x\in \bigcap \Sigma_f)}=1_{\B}.
\end{equation}
We give a proof of this equality for the sake of completeness:
\begin{proof}[Proof of equation (\ref{eqn:XXX})]
\begin{claim}\label{claim1}
Let $G$ be a $V$-generic filter for $\B$. Then:
\[
V[G]\models \exists! x \cp{x\in \bigcap \Sigma_{f}^G}
\]
where $\Sigma_{f}^G=\set{\cl{U_n}^{V[G]}:a_n \in G }$.
\end{claim}
\begin{proof}[Proof of the claim]
The preimage of $\C$
through $f$ contains
an open dense subset of $St(\B)$ in $V$, hence it follows that
\[
D=\set{a_n=\intclosure{f^{-1}[U_n]}: n\in\omega}\in V
\]
is a predense subset of $\B^+$. 
Since $G$ is $V$-generic, $G\cap D \not= \emptyset$.
Thus $a_m \in G$ and
$\cl{U_m}^{V[G]} \in \Sigma_{f}^G$ for some $m\in\omega$.
The proof that $\bigcap \Sigma_{f}^G$ is a singleton can be carried out
as in Claim \ref{claim}.
\end{proof}
The Claim holds for all $V$-generic filters $G$ for $\B$, thus
\[
\Qp{\exists! x (x\in \bigcap \Sigma_{f})}=1_{\B}.
\]
completing the proof of equation (\ref{eqn:XXX}).
\end{proof}
$V^{\B}$ is full, hence there is a $\B$-name
$\tau_f$ such that
\[
\Qp{\tau_f \in \bigcap \Sigma_f} = 1_{\B}.
\]
This is a $\B$-name for a complex number.
Moreover, if $\tau$ is a $\B$-name for a complex number and
\[
\Qp{\tau \in \bigcap \Sigma_f} = 1_{\B},
\]
then, from
\[
(\tau_f \in \bigcap \Sigma_f)\wedge(\tau \in \bigcap \Sigma_f)\wedge(\exists! x
 (x\in \bigcap \Sigma_f))
\rightarrow \tau = \tau_f
\]
it follows that:
\[
\Qp{\tau=\tau_f}=1_{\B}.
\]
This shows that the map $f\mapsto\tau_f$ can be defined.

To conclude the proof of Proposition~\ref{prop:iso} we still must show that
\begin{equation}\label{eqn:YYY}
\Qp{\tau_f\in\dot{U}_n}=\intclosure{f^{-1}[U_n]}=a_n
\end{equation} 
\begin{proof}[Proof of equation (\ref{eqn:YYY})]
Let $G$ be a
$V$-generic filter for $\B$.
On the one hand we have
(using the same proof of the uniqueness part
in Claim \ref{claim}) that if $(\tau_{f})_G \in U_n^{V[G]}$
then $a_m \in G$ for some $m$ such that $(\dot{U}_n,a_m)\in \Sigma_f$ and $a_m\in G$,
which necessarily gives that $m=n$, obtaining
\[
\Qp{\tau_f \in \dot{U}_n} \le a_n.
\]
On the other hand
\[
G \in {f^{V[G]}}^{-1}[U_n^{V[G]}] \Rightarrow
(\tau_f)_G=^{by~\ref{extension}} f^{V[G]}(G)\in U^{V[G]}_n \Rightarrow
\Qp{\tau_{f} \in \dot{U}_n}\in G
\]
which means, interpreting $\Qp{\tau_{f} \in \dot{U}_n}$ as a clopen
subset of $St(\B)^{V[G]}$, that
\[
{f^{V[G]}}^{-1}[U_n^{V[G]}] \subseteq \cp{\Qp{\tau_{f} \in \dot{U}_n}}^{V[G]}.
\]
Lemma \ref{borel} guarantees that this is equivalent to
\[
f^{-1}[U_n^V] \subseteq \Qp{\tau_{f} \in \dot{U}_n}.
\]
Since $\Qp{\tau_{f} \in \dot{U}_n}$ is clopen, this implies that
\[
a_n=\intclosure{f^{-1}[U_n^V]}\le\Qp{\tau_{f}\in \dot{U}_n}
\]
\end{proof}
Proposition~\ref{prop:iso} is proved.
\end{proof}

\begin{proof}[Proof of Lemma~\ref{surj}]
The proof that $\Sigma_\tau^H$ is non-empty iff its intersection has
one single point can be carried out as in Claim \ref{claim}
substituting all over the proof $\intclosure{f^{-1}[U_n]}$
with $\Qp{\tau \in \dot{U}_n}$.

\begin{description}
\item[\rm \underline{Preimage of $\set{\infty}$ is nowhere dense:}]
We show that the preimage of $\C$ through
$f_{\tau}$ contains an open dense set. Set
\[
a_n=\Qp{\tau \in \dot{U}_n}
\]
and consider the set $A=\set{a_n:n \in \omega}$. We show that:
\[
\bigvee_{n\in \omega} a_n =1_{\B}.
\]

Since $\tau$ is a $\B$-name for a complex number in $M$, if
$G$ is a $V$-generic filter over $\B$ we
have:
\[
V[G]\models \tau^G \in \C^{V[G]}
\]
We can thus infer
\[
V[G]\models \exists n\in\omega (\tau^G \in U_n)
\]
for all $V$-generic filters $G$, since $\C^{V[G]}\cap V[G]=\bigcup_{n\in\omega} U_n^{V[G]}\cap V[G]$.
Thus:
\[
\bigvee_{n\in \omega}a_n =\Qp{\exists n \in \check{\omega}(\tau \in \dot{U}_n)}
\ge 1_{\B}
\]
This implies that $A$ is predense and therefore that $\bigcup_{n \in \omega} \mathcal{O}_{a_n}$ is
dense in $St(\B)$.
\item[\rm \underline{Continuous:}] Let $H\in St(\B)$ be
in the preimage of $\C$, and let $U$ be an open subset
of $\C$ containing $f_{\tau}(H)$. Consider
$U_k\in\mathcal{B}$ such that
\[
f_{\tau}(H) \in U_k
\]
\[
\cl{U_k}\subseteq U
\]
Since
\begin{equation}\label{formula}
f_{\tau}(H) \in U_k \Rightarrow a_k \in H, \tag{1}
\end{equation}
(this can be proved as in
the uniqueness part in Claim \ref{claim} substituting
$\intclosure{f^{-1}[U_n]}$
with $\Qp{\tau \in \dot{U}_n}$), and since
the following inclusion holds
\[
\mathcal{O}_{a_k}\subseteq f_{\tau}^{-1}(U),
\]
the continuity of $f_{\tau}$ for points in the preimage of $\C$ is proved.

Consider now $H \in f_{\tau}^{-1}(\set{\infty})$. Let $A$ be an open
neighborhood of $\infty$, and let $U_k\in \mathcal{B}$ be such that:
\[
\cl{U_k}^c \subseteq A
\]
We also consider $U_l$ such that
\[
\cl{U_k} \subseteq U_l
\]
By definition of $f_{\tau}$ we have that $H\in \mathcal{O}_{a_l}^c$, and by
equation (\ref{formula}) the image of any element in the open set $\mathcal{O}_{a_l}^c$ cannot belong
to $U_l$. Thus
\[
\mathcal{O}_{a_l}^c \subseteq f_{\tau}^{-1}[U_l^c]
\subseteq f_{\tau}^{-1}[\cl{U_k}^c]\subseteq f_{\tau}^{-1}[A]
\]
\item[\rm \underline{$\tau_{f_{\tau}}=\tau$:}]
We already know that (see equation (\ref{formula})):
\[
f^{-1}_{\tau}[U_n] \subseteq \mathcal{O}_{a_n}
\]
The second set is clopen, therefore:
\begin{align}
\Qp{\tau_{f_{\tau}} \in \dot{U}_n}= \intclosure{f^{-1}_{\tau}[U_n]}
\subseteq \mathcal{O}_{a_n} \tag{2}\label{f2}
\end{align}
Toward a contradiction, assume $\Qp{\tau = \tau_{f_{\tau}}}\not= 1_{\B}$
and let $G$ a $V$-generic filter which verifies
\[
V[G]\models\tau_G \not= (\tau_{f_{\tau}})_G
\]
Thus there exists $n\in \omega$
such that:
\[
(\tau_{f_{\tau}})_G \in U_n^{V[G]}
\]
\[
\tau_G \notin U_n^{V[G]}
\]
The inclusion relation (\ref{f2}) implies
\[
\Qp{\tau_{f_{\tau}}\in \dot{U}_n} \le a_n=\Qp{\tau\in \dot{U}_n}
\]
but by Cohen's Forcing Theorem $\Qp{\tau_{f_{\tau}} \in \dot{U}_n}\in G$. This is a contradiction.
\end{description}
The Lemma is proved.
\end{proof}

\begin{proof}[Proof of Lemma~\ref{lem:embed}]
We will consider in detail the case of $R \subseteq \C$ a unary
Borel relation in $\C$, the general case for $n$-ary $R$ is immediate. 
Given $f\in C^+(St(\B))$, consider
$\Qp{R(f)}$ and $\Qp{\tau_f \in \dot{R}}$ as regular open
subsets of $St(\B)$. In order to show that
they are equal, it is sufficient to prove that their symmetric difference
is meager. By definition, we already know that $\Qp{R(f)}$ has meager
difference with the set
\[
\set{H\in St(\B):f(H)\in R}=f^{-1}[R].
\]
Therefore it suffices to prove that $\Qp{\tau_f \in \dot{R}}$ and $f^{-1}[R]$
have meager difference.
The proof proceeds step by step on the hierarchy of Borel sets $\Sigma^0_{\alpha}$,
$\Pi_{\alpha}^0$, for $\alpha$ a countable ordinal.
\begin{description}
\item[\rm \underline{$\Sigma^0_1$:}] Let $R$ be an element of the basis
\[
\mathcal{B}=\set{U_n:n\in \omega}
\]
defined in Remark \ref{code}. The thesis follows from Proposition \ref{prop:iso}, in fact
\[
\Qp{\tau_f \in \dot{U}_n}=\intclosure{f^{-1}[U_n]}
\]
which has meager difference with $f^{-1}[U_n]$. Consider now
\[
R=\bigcup_{i\in \mathcal{I}}U_i
\]
where $\mathcal{I}$ is a countable set of indexes. In this case
we have that
\[
f^{-1}[R]=\bigcup_{i \in \mathcal{I}}f^{-1}[U_i]
\]
and
\[
\Qp{\tau_f \in \dot{R}}=\bigvee_{i \in \mathcal{I}}\Qp{\tau_f \in \dot{U}_i}=
\intclosure{A}
\]
where $A=\bigcup_{i\in\mathcal{I}}\Qp{\tau_f \in \dot{U}_i}$.
For each $i \in \mathcal{I}$, the sets $f^{-1}[U_i]$ and $\Qp{\tau_f \in \dot{U}_i}$
have meager difference, thus $
f^{-1}[R] \Delta A$
is meager.
The proof is therefore concluded because $A
\Delta  \intclosure{A}$
is meager.
\item[\rm \underline{$\Sigma^0_{\alpha} \Rightarrow \Pi^0_{\alpha}$:}] Suppose $R\in \Pi_{\alpha}^0$,
and that the thesis holds for Borel sets in $\Sigma_{\alpha}^0$.
By definition $R^c \in \Sigma_{\alpha}^0$, therefore:
\[
f^{-1}[R^c]\Delta \Qp{\tau_f \in \dot{R}^c} \text{ is meager}
\]
hence
\[
f^{-1}[R] \Delta \Qp{\tau_f \in \dot{R}} \text{ is meager}
\]
\item[\rm \underline{$\Pi^0_{\alpha}\Rightarrow \Sigma^0_{\alpha+1}$:}] This item
can be proved as the second part of the case $\alpha=1$, substituting the $U_n$ with
Borel sets in $\Pi^0_{\alpha}$.
\item[\rm \underline{$\Sigma^0_{\beta}$ for $\beta$ limit ordinal:}] If the thesis holds for
$\alpha < \beta$, then the proof can be carried similarly to the case $\Pi^0_{\alpha} \Rightarrow
\Sigma^0_{\alpha+1}$.
\end{description}
The Lemma is proved.
\end{proof}

This concludes the proof of Theorem~\ref{thm:iso} for the case $Y=\C$.

\subsection{$C(St(\B))/G$ and $C^+(St(\B))/G$ in generic extensions}
The following  proposition shows that if we restrict our attention to $V$-generic filters for $\B$ then 
$C(St(\B))$ is a family of names large enough to describe all complex numbers of $V[G]$.
\begin{proposition} \label{gener}
Assume $V$ is a model of $\ZFC$, $\B$ a complete boolean algebra in $V$ and $G$
a $V$-generic filter in $\B$. Then
\[
C^+(St(\B))/G \cong C(St(\B))/G
\]
\end{proposition}
\begin{proof}
We need to show that for each $f \in C^+(St(\B))$ we can find an $\tilde{f} \in C
(St(\B))$ such that
\[
\Qp{f = \tilde{f}} \in G
\]
which, by Corollary \ref{extension}, is equivalent to
\[
f^{V[G]}(G)=\tilde{f}^{V[G]}(G)
\]
We denote again
\[
a_n=\intclosure{f^{-1}[U_n]}.
\]
Proceeding as in Claim \ref{claim1}, we can find 
$m\in\omega$ such that $a_m \in G$. For each $H\in \mathcal{O}_{a_m}$ we have that
\[
f(H)\in \cl{U_m}
\]
by Lemma \ref{extfunction}.
We can therefore consider the restriction of $f$ to $\mathcal{O}_{a_m}$ (which is clopen)
and extend it to a $\tilde{f}\in C
(St(\B))$ setting it to be constantly $0$ on $\mathcal{O}_{\neg a_m}$.
The implication
\[
f\restriction_{\mathcal{O}^V_{a_m}}=\tilde{f}\restriction_{\mathcal{O}^V_{a_m}}
\Rightarrow
f^{V[G]}\restriction_{\mathcal{O}^{V[G]}_{a_m}}=\tilde{f}^{V[G]}\restriction_{\mathcal{O}^{V[G]}_{a_m}}
\]
guarantees the thesis, since
$G\in \mathcal{O}^{V[G]}_{a_m}$.
\end{proof} 

\subsection{Proof of Theorem  \ref{thm:iso} for an arbitrary Polish space $Y$}
We outline the proof of Theorem  \ref{thm:iso} for the case of an arbitrary Polish space $Y$.
The strategy of the proof is exactly the same for the case $Y=\C$. 
At some points the corresponding Lemma needs a slightly 
more elaborate proof, we outline when this is the case.
All over this section let 
\[
Y=\bigcap_{n\in\mathbb{N}}\bigcup\set{B_{r_{mn}}(q_{mn}):m\in\mathbb{N}}
\subseteq \mathcal{H}
\] 
denote an arbitrary Polish space seen as a $G_\delta$ subset of $\mathcal{H}$
and 
$\set{U_n:n\in\omega}$ denote its basis 
\[
\hat{\mathcal{B}}_Y=\set{B_r(q)\cap Y:r\in\mathbb{Q}, q\in D}
\]
as done in section~\ref{section:names}.
\begin{lemma} \label{extfunction1}
Assume $f \in V$ is an element of $C^+(St(\B),Y)$.
For $H\in St(\B)$ we define
\[
\Sigma^H_f=\set{\cl{U_n}:
\intclosure{f^{-1}[U_n]} \in H}
\]
Then,
for $H\in St(\B)$ such that $f(H)\in Y$, we have:
\[
\set{f(H)}= Y\cap \bigcap\Sigma_f^H,
\]
moreover $\bigcap\Sigma_f^H$ is always non-empty.
\end{lemma}
\begin{proof}
$\bigcap\Sigma_f^H$ is always non-empty, since it is the intersection of a family
with the finite intersection property of closed sets of a compact space.
The proof that $|\bigcap\Sigma_f^H\cap Y|\leq 1$ if $f(H)\in Y$ runs as the uniqueness part of Lemma~\ref{extfunction}. 
\end{proof}

\begin{proposition} \label{prop:iso1}
Fix $V$ a transitive model of $\ZFC$ and $\B\in V$ a boolean algebra which 
$V$ models to be complete.
Let $f \in C^+(St(\B),Y)$.
For
each $n\in\omega$ let
\[
a_n=\intclosure{f^{-1}[U_n]}
\]
where $\mathcal{B}_Y=\set{U_n:n\in\omega}$ is the fixed countable basis for $Y$.
There exists a $\tau_f \in Y^{\B}$ such that\footnote{If $U_n=Y\cap B_r(q)$ and $G$ is 
$V$-generic for $\B$,
$\dot{U_n}$ denotes the $\B$-name for the elements in
the Hilbert cube of $V[G]$ belonging to 
\[
\bigcap_{n\in\mathbb{N}}\bigcup\set{(B_{r_{mn}}(q_{mn}))^{V[G]}:m\in\mathbb{N}}
\cap (B_r(q))^{V[G]},
\]
where $(B_{r}(q))^{V[G]}$ is the ball in the Hilbert cube $\mathcal{H}^{V[G]}$
of rational radius $r$ and center $q$ as computed in $V[G]$.}
\[
\Qp{\tau_f \in \dot{U}_n}^{V^\B}=a_n
\]
for all $n\in\omega$.
Moreover any $\tau\in V^\B$ such that 
$\Qp{\tau\in \dot{U}_n}^{V^\B}=a_n$
for all $n\in\omega$ is also such that $\Qp{\tau=\tau_f}^{V^\B}=1_\B$.
\end{proposition}
\begin{proof}
This proposition has exactly the same proof as the corresponding Proposition~\ref{prop:iso} for $\C$.
\end{proof}
\begin{lemma}\label{surj1}
Assume $\tau \in Y^{\B}$. Consider
\begin{align*}
f_{\tau}:St(\B) &\to \mathcal{H} \\
H &\mapsto \sigma_\tau^H
\end{align*}
where, given
\[
\Sigma_\tau^H=\set{\cl{U_n}:\Qp{\tau\in \dot{U}_n}^{V^\B}\in H},
\]
$\sigma_\tau^H $ is the unique element
in $\bigcap \Sigma_\tau^H$ if $\Sigma_\tau^H$ is non-empty.
Otherwise $f(H)=\sigma_\tau^H$ is defined by extending by continuity $f$
on the others $H\in St(\B)$.
The function $f_\tau$ belongs to $C^+(St(\B),Y)$
and $\tau_{f_{\tau}}=\tau$.
\end{lemma}
\begin{proof}
Notice that for all $n\in\omega$
\[
\Qp{\tau\in\bigcup\set{\dot{B}_{r_{mn}}(q_{mn}):m\in\mathbb{N}}}=1_\B,
\]
hence for all $n\in\omega$
\[
\set{\Qp{\tau\in\dot{B}_{r_{mn}}(q_{mn})}:m\in\omega}
\]
is predense in $\B^+$.
Seeing each $\Qp{\tau\in\dot{B}_{r_{mn}}(q_{mn})}$ as a clopen subset of $St(\B)$, we conclude that
\[
A_n=\bigcup\set{\Qp{\tau\in\dot{B}_{r_{mn}}(q_{mn})}:m\in\omega}
\]
is open dense in $St(\B)$ for all $n\in\omega$.
Hence $f_\tau$ is well defined (and continuous) on the dense $G_\delta$ subset of $St(\B)$ $\bigcap_{n\in\omega}A_n$.
Therefore $f$ can be extended by continuity to the whole of $St(\B)$.
The proof of the continuity of $f$ on $\bigcap_{n\in\omega}A_n$, and the fact that on
$\bigcap_{n\in\omega}A_n$ $f$ takes values in $Y$ can be carried as in the corresponding 
proof of Lemma~\ref{surj}. 
\end{proof}

The proof that 
\begin{equation}\label{eqn:ZZZY}
\Qp{f=g}^{C^+(St(\B),Y)}=\Qp{\tau_f=\tau_g}^{V^\B}
\end{equation}
for any $f,g\in C^+(St(\B),Y)$
is the same as the corresponding proof for equation~\ref{eqn:ZZZ}.

\subsection{Extensions of the boolean isomorphism}
In general any boolean predicate or function on the $\B$-valued model $C^+(St(\B),Y)$ can be transferred to a 
corresponding boolean predicate on $Y^\B$ using the above isomorphism $f\mapsto \sigma_f$.
\begin{definition}
Let $Y$ be a Polish space and $\B$ a complete boolean algebra.
For any boolean relation $R^{St(\B)}: C^+(St(\B),Y)^n\to\B$ (and boolean
function $F^{St(\B)}$)
\[
R^{\B}(\sigma_1,\dots,\sigma_n)=R^{St(\B)}(f_{\sigma_1},\dots,f_{\sigma_n}),
\]
(similarly we can define the boolean function $F^\B$).
\end{definition}

By Theorem~\ref{thm:iso} and Lemma~\ref{lem:embed}, we immediately have the following.
\begin{theorem}\label{thm:iso2}
Fix a signature 
\[
\mathcal{L}=\set{R_i:i\in I}\cup \set{F_j:j\in J}.
\]
and assume that $\langle R_i^{St(\B)}:i\in I \rangle$, $\langle F_j^{St(\B)}:j\in J\rangle$ are boolean interpretations 
of the signature making $C^+(St(\B),Y)$ a $\B$-valued model.
The map
\begin{align*}
\Gamma: C^+(St(\B),Y) &\to Y^{\B} \\
f &\mapsto\tau_f
\end{align*}
is an isomorphism of the $\B$-valued model
\[
\langle C^+(St(\B),Y), R_i^{St(\B)}:i\in I, F_j^{St(\B)}:j\in J\rangle
\]
with the $\B$-valued model
\[
\langle Y^\B, R_i^{\B}:i\in I, F_j^{\B}:j\in J\rangle.
\]
\end{theorem}

\subsection{Some further comments on the proof of Theorem \ref{thm:iso}}
One can get a proof of this theorem for the case $Y=\C$ following
Jech's\footnote{This isomorphism of $\C^\B$ and $C^+(St(\B))$ 
has also been independently proven by Ozawa in
\cite{MR759416}, but Jech's proof is in our eyes more elegant and informative.} results in~\cite{MR779056}
as follows: Jech defines the notion of stonean algebra as 
an abelian space of (possibly unbounded) normal operators. Stonean algebras are a
natural generalization of the notion of commutative $C^*\!$-algebras. Jech proves that: 
\begin{itemize}
\item
The isomorphism type of any \emph{complete} 
stonean algebra is determined
by the complete boolean algebra given by its space of \emph{projections}, 
\item
For any complete boolean algebra
$\B$, $\C^\B$ and $C^+(St(\B))$ are  \emph{complete} 
stonean algebras whose spaces of \emph{projections} are in both 
cases isomorphic to $\B$.
\end{itemize}
Jech's proof that $\C^\B$ is a complete stonean algebra exploits the property that 
$(\mathbb{R},<)$ is a complete linear order in order to give a simple description of the 
$\B$-names for real numbers of $V^\B$, and also the property that any element of a stonean algebra can be 
decomposed uniquely as the direct sum of its real and imaginary part.
The isomorphism between $\C^\B$ and $C^+(St(\B))$ is obtained by
showing that the Gelfand transform can be defined also for
stonean algebras and yields that any stonean algebra $\mathcal{A}$ 
is isomorphic to $C^+(X)$ where $X$ is the spectrum of $\mathcal{A}$. Moreover, in case $\mathcal{A}$ is complete,
its spectrum $X$ is also homeomorphic to the Stone space
of the complete boolean algebra of projections on $\mathcal{A}$.
In both arguments there are peculiar properties of $\mathbb{R}$ (being a complete linear order) 
which are not shared by other Polish spaces $Y$, and
of a stonean algebra $\mathcal{A}$ (the characterization of its elements 
in terms of the involution operation and of its self-adjoint operators)
which are not shared by the function spaces $C^+(X,Y)$ for $Y\neq \C$. 
Ozawa's proof relies on the same properties of $\mathbb{R}$ and of commutative algebras of normal operators
used in Jech's argument.
In particular we do not see any natural pattern to generalize
Jech's (or Ozawa's) proof method 
so to cover also the cases of Theorem \ref{thm:iso} for a Polish space $Y\neq \C$ other than resorting
(as we did) to purely topological characterizations of the properties of Polish spaces.
A further comment is in order at this point: we became aware of Jech's and Ozawa's work only 
after having completed 
and submitted a first version of this paper.

\section{Generic absoluteness}\label{sec5}

We can now show that for any Polish space $Y$
the $\B$-valued models $(C^+(St(\B),Y),R^{St(\B)})$, 
with $R$ a Borel (universally Baire) relation on $Y^n$,
is an elementary superstructure of $(Y,R)$.
By Lemma~\ref{lem:embed}, whenever $R$ is a Borel relation on $Y^n$ with 
$Y$ Polish, $R^\B(\sigma_1,\dots,\sigma_n)=R^{St(\B)}(f_{\sigma_1},\dots,f_{\sigma_n})$ (where $R^\B$ is defined as
in Def.~\ref{code}). This equality is a special case of the much more general result which can be
proved for universally Baire relations.

\begin{definition}[Feng, Magidor, Woodin~\cite{FENMAGWOO}]
Let $Y$ be a Polish space. 
$A\subseteq Y^n$ is universally Baire if $f^{-1}[A]$ has the Baire property in $Z$
for all continuous $f:Z\to Y^n$ and all compact Hausdorff spaces $Z$.
\end{definition}
$\mathsf{UB}$ denote the class of universally Baire subsets of $\mathcal{H}$ (or any other Polish space).

\begin{fact}
Let $Y$ be a Polish space. 
$A\subseteq Y^n$ is universally Baire if and only if $f^{-1}[A]$ has the Baire property in $Z$
for all continuous $f:Z\to Y^n$ with $Z$ compact and extremally disconnected.
\end{fact}
\begin{proof}
We need to prove just one direction, and we prove it as follows.
Assume $f:Z\to Y^n$ is continuous for some $Z$ compact Hausdorff but not extremally disconnected.
Set
$Z^*=St(\RO(Z))$ and define $\pi:Z^*\to Z$ by 
$\pi(G)=x$ if $x$ is the unique point in $Z$ belonging to 
\[
\Sigma^G=\bigcap\set{\cl{U}: U\in G}.
\]
The same arguments we encountered in the proof of the isomorphism of $C^+(St(\B))$ with $\C^\B$ show that
$\pi$ is continuous, open and surjective.
In particular $f^{-1}[A]$ has the Baire property in $Z$ iff $g^{-1}[A]$ has the Baire property in $Z^*$, 
where $g=f\circ \pi$. 
\end{proof}

By \cite[Chapter $29$, Lemma $5$]{halmos} 
Borel sets are universally Baire as already observed in Example~\ref{exmp:key}.
Woodin \cite[Theorem 3.4.5, Remark 3.4.7]{LARSON} showed that for any universally Baire set $A$
the theory of $L(\mathbb{R},A)$ is 
generically invariant in the presence
of class many Woodin cardinals which are a limit of Woodin cardinals, and moreover 
that these assumptions entail that any $\Sigma^1_n$-property defines a universally Baire relation. 
Shoenfield \cite[Lemma 25.20]{Jech}
(or \cite[Theorem 3.5.3, Remark 3.5.4]{tesi} or \cite[Lemma 1.2]{MR3453587} for a presentation of this result in line with
the content of this paper) showed that the $\Sigma^1_2$-theory of
any Polish space $X$ is generically invariant under set forcing.
This translates by the results of this paper in the following:
\begin{theorem}
Assume
$\langle R_i:i\in I\rangle$ and $\langle F_j:j\in J\rangle$ are Borel predicates and functions on some Polish space $Y$.
Let $X$ be a compact Hausdorff extremally disconnected space and $p\in X$. Then
\[
\langle Y,R_i:i\in I,F_j:j\in J\rangle\prec_{\Sigma_2}\langle C^+(X,Y)/p,R^X_i/p:i\in I,F^X_j/p:j\in J\rangle.
\]
Moreover if we assume the existence of class many Woodin cardinals which are a limit of Woodin cardinals, 
we can let each $R_i$ and $F_j$ be 
arbitrary universally Baire relations and functions, and we have the stronger conclusion that
\[
\langle Y,R_i:i\in I,F_j:j\in J\rangle\prec\langle C^+(X,Y)/p,R^X_i/p:i\in I,F^X_j/p:j\in J\rangle.
\]
\end{theorem}

\begin{proof}
By Shoenfield's (or Woodin's) theorem we have that for all $\Sigma^1_2$ ($\Sigma^1_n$ for any $n$)
properties $\phi(\vec{x})$ in the parameters $\langle R_i:i\in I\rangle$, $\langle F_j:j\in J\rangle$
with each $R_i,F_j$ Borel (universally Baire)
the following are equivalent:
\begin{itemize}
\item
$\phi(\vec{r})$ holds in $\langle Y,R_i:i\in I,F_j:j\in J\rangle$,
\item
$\Qp{\phi(\vec{r})}=1_\B$ in $\C^{\B}$ for some complete boolean algebra $\B$,
\item 
$\Qp{\phi(\vec{r})}=1_\B$ in $\C^{B}$ for all complete boolean algebras $\B$.
\end{itemize}
Since $X$ is compact Hausdorff and extremally disconnected, $\mathsf{CL}(X)$ is a complete boolean algebra
and $X$ is homeomorphic to $St(\mathsf{CL}(X))$.
By Theorem \ref{thm:iso2} $C^+(X,Y)$ and $Y^{\B}$ are isomorphic $\B$-valued models. 
In particular $C^+(X,Y)$ is full.
By the first two equivalent items we get that 
$\Qp{\phi(\vec{r})}^{C^+(X,Y)}=1_\B$ in $C^+(X,Y)$ if and only if $\phi(\vec{r})$ holds in $Y$. 
Since the above holds for all relevant properties $\phi$, we can apply \L o\'{s}'s theorem to the full $\B$-valued model
$C^+(X,Y)$ in the point (ultrafilter) $p$ to conclude that
\[
\langle Y,R_i:i\in I,F_j:j\in J\rangle\prec _{(\Sigma_2)} \langle C^+(X,Y)/p,R^X_i/p:i\in I,F^X_j/p:j\in J\rangle.
\]
\end{proof}

Following Takeuti's ideas, we remark 
that these results suggest the following ``original" proof strategy to be applied in an 
algebraic geometric context rather than in an operator algebra context (as already done by Takeuti and others). 
Prove that a certain problem
regarding for example complex numbers and analytic functions 
has a solution in some forcing extension. Then argue that its solution 
can be formalized as a first order property of the structure
$C^+(X)/p$. Conclude using elementarity that the solution of the problem for the complex numbers 
is really the one computed in $C^+(X)/p$.
We have already successfully applied the above strategy to prove a result related to Schanuel's conjecture
in number theory (unfortunately for us
already proved by other means): the interested reader is referred to \cite{VIASCH15}.

\subsection*{Acknowledgements} 
The second author acknowledges support from the PRIN2012 Grant ``Logic, Models and Sets"
(2012LZEBFL),
and the Junior PI San Paolo grant 2012 NPOI (TO-Call1-2012-0076).
This research was completed whilst the second author was a visiting fellow at the 
Isaac Newton Institute for Mathematical Sciences in the programme ``Mathematical, 
Foundational and Computational Aspects of the Higher Infinite'' (HIF) funded by EPSRC grant EP/K032208/1.

	\bibliographystyle{amsplain}
	\bibliography{Bibliography}
\end{document}